\documentclass[11.5pt]{amsart}
\usepackage[letterpaper]{geometry}
\usepackage{amsmath,amsfonts,amsthm,amscd,amssymb,mathrsfs,amssymb, multirow, pst-node, comment,enumitem}
\usepackage{stmaryrd}
\usepackage{graphicx}
\usepackage{verbatim}
\usepackage{mathtools}
\usepackage{tikz-cd} 
\newtheorem{theorem}{Theorem}

\newtheorem{proposition}[theorem]{Proposition}
\newtheorem{lemma}[theorem]{Lemma}

\numberwithin{figure}{section}
\usepackage[colorlinks, linkcolor=blue, anchorcolor=black, citecolor=red]{hyperref}
\theoremstyle{definition}

\newtheorem{remark}[theorem]{Remark}

\newcommand{\ddc}{{\rm dd^c}}

\renewcommand{\max}{{\rm max}}

\newcommand{\tr}{{\rm tr}}

\newcommand*{\sheafhom}{\mathscr{H}\kern -.5pt om}
\newcommand*{\sheafext}{\mathscr{E}\kern -.5pt xt}

\newcommand{\PSH}{{\rm PSH}}

\makeatletter
\newcommand{\colim@}[2]{%
  \vtop{\m@th\ialign{##\cr
    \hfil$#1\operator@font colim$\hfil\cr
    \noalign{\nointerlineskip\kern1.5\ex@}#2\cr
    \noalign{\nointerlineskip\kern-\ex@}\cr}}%
}
\newcommand{\colim}{%
  \mathop{\mathpalette\colim@{\rightarrowfill@\scriptscriptstyle}}\nmlimits@
}
\renewcommand{\varprojlim}{%
  \mathop{\mathpalette\varlim@{\leftarrowfill@\scriptscriptstyle}}\nmlimits@
}
\renewcommand{\varinjlim}{%
  \mathop{\mathpalette\varlim@{\rightarrowfill@\scriptscriptstyle}}\nmlimits@
}
\makeatother

\renewcommand{\tilde}{\widetilde}

\numberwithin{equation}{section}
\numberwithin{theorem}{section}
\numberwithin{table}{section}
\numberwithin{table}{section}

\begin{document}
\bibliographystyle{amsalpha}
\title{Parabolic Monge-Amp\`ere equations on a family of Calabi-Yau manifolds}

\author{Jiyuan Han}
\address{Institute for Theoretical Sciences,
Westlake University,
No.600 Dunyu Road, 
Hangzhou, 310030, China}
\email{hanjiyuan@westlake.edu.cn}

\thanks{}
\date{\today}
\begin{abstract}
We prove a uniform $L^\infty$-estimate for the parabolic Monge--Amp\`ere equation under uniform H\"ormander-Tian type estimates with respect to the initial and prescribed measures. We apply this estimate to rescaled K\"ahler--Ricci flows on a family of polarized Calabi--Yau manifolds over the punctured disk. An interpolation argument gives uniform H\"ormander-Tian type estimates along the flows. We also establish uniform diameter bounds and Gromov--Hausdorff precompactness for this family of flows. In particular, we obtain an alternative proof of the diameter estimate by Li-Tosatti in \cite{LT2024}.
\end{abstract}
\maketitle
\setcounter{tocdepth}{1}

\section{Introduction}
\label{introduction}

The degeneration of polarized Calabi-Yau manifold over a punctured disk has been extensive studied by many experts. See \cite{GW00, KS04, NXY19, KX20, HSVZ22, LiYang22, LiYang24, LT2024, HJMM24, SZ24, LiYang25, TO25, LiYang26, BL26} for an incomplete list. In such a degeneration, the induced Fubini--Study metrics are closely related to a choice of model, while the Ricci-flat metrics describe the geometry of the Calabi--Yau fibers. This motivates studying an interpolation between these two families of metrics by K\"ahler--Ricci flows (parabolic Monge--Amp\`ere equations in the level of potential functions).

In the first part of the paper,
we will consider the case when $X$ is a general compact K\"ahler manifold of complex dimension $n$. Let $\omega$ be a K\"ahler form on $X$.
Let $\mu$ be a Radon measure on $X$ with a normalization $\int_X\mu = 1$. 

We say there is a uniform H\"ormander-Tian
type estimate with respect to the metric $\omega$ and the measure $\mu$ if there exist fixed constants $\alpha, A>0$, such that
\begin{align}
    \label{uniform_Skoda}
    \int_X e^{-\alpha(u-\sup_X u)}\; \mu < A
\end{align}
for any $u\in \PSH(X,\omega)$.

In the following theorem, we will prove the uniform estimates and the smooth convergence of the solution to the Monge-Amp\`ere flow without assuming the existence of the solution to the stationary Monge-Amp\`ere equation $  \frac{1}{V_\omega}(\omega+\ddc\varphi)^n  =e^F\mu $.
\begin{theorem}
    \label{main_thm_1}
    Let $X$ be a compact K\"ahler manifold of dimension $n$. Let $\omega$ be a K\"ahler form on $X$. Let $\mu$ be a Radon probability measure on $X$ with smooth positive density function. Assume that both $\frac{1}{V_\omega}\omega^n$ and $\mu$ satisfy the uniform H\"ormander-Tian inequality \eqref{uniform_Skoda}.
    Let $F\in C^\infty(X)$, $e^{F}\in L^1\log^pL(\mu)$, $p>n+1$, and $\int_X e^F\mu = 1$. Consider the following parabolic Monge-Amp\`ere flow
    \begin{align}
    \label{parablic_MA}
        \begin{cases}
            e^{\partial_t\varphi}\;e^F\mu = \frac{1}{V_\omega}(\omega+\ddc\varphi)^n, \\
            \varphi(0) = 0,
        \end{cases}
    \end{align}
    where $V_\omega = \int_X \omega^n$.
    Then there exists a constant $C = C(n, \alpha, p, A,  K:=\|e^F\|_{L^1\log^p(\mu)})$ such that 
    \begin{align*}
        \|\varphi\|_{L^\infty} \leq C
    \end{align*}
    on $X\times [0,+\infty)$.

    In addition, $\varphi$ converges to $\varphi_\infty\in C^\infty(X)$ smoothly, where $\varphi_\infty$ is the solution of
    \begin{align*}
        \frac{1}{V_\omega}(\omega+\ddc\varphi_\infty)^n = e^F\mu.
    \end{align*}
\end{theorem}
\begin{remark}
    Inequality \eqref{uniform_Skoda} has a close relation with the algebraic properties of $(X,\omega)$. In particular, when $\mu = \frac{\omega^n}{V_\omega}$, the upper threshold of $\alpha$ is called the $\alpha$-invariant of $(X,\omega)$ (defined by Tian in \cite{Tian87}) , which plays an important role in K-stability theory. More details on the algebraic side can be found in \cite{Xu2025KStability}.
\end{remark}

In the second part of the paper, 
we will apply Theorem \ref{main_thm_1} to obtain some uniform analytic and geometric properties of the K\"ahler Ricci flows on a one parameter family of Calabi-Yau manifolds. Let $\mathbb{D}$ denote the closed disk of radius $e^{-1}$.
Let $\pi: X \rightarrow \mathbb{D}^*$ be a family of degeneration of polarized Calabi-Yau manifolds, equipped with a relatively ample $\mathbb{Q}$-line bundle $L$ on $X$. By \cite{KKMSD73}, up to taking a base change, there exists a semi-stable reduction
$\pi: \mathcal{X}\rightarrow \mathbb{D}$, with a relatively ample $\mathbb{Q}$-line bundle $\mathcal{L}$ on $\mathcal{X}$. And the restriction of $(\mathcal{X},\mathcal{L})$ over $\mathbb{D}^*$ is equal to $(X,L)$. 
Set $L_s = L|_{X_s}$ for $s\in \mathbb{D}^*$.
For simplicity, let's normalize the total volume by
\begin{align*}
    L^n\cdot [X_s] = 1.
\end{align*}
We can choose a sufficiently divisible integer $m$, such that $m\mathcal{L}$ is a relatively very ample line bundle on $\mathcal{X}$, which implies an embedding
\begin{align*}
    \iota: \mathcal{X} \xrightarrow{|m\mathcal{L}|} \mathbb{P}^N\times \mathbb{D}.
\end{align*}
Let $\omega_{\rm FS}$ denote the Fubini-Study metric on $\mathbb{P}^N$ and $\omega_{\mathbb{D}}$ the standard flat metric on $\mathbb{D}$. By choosing a sufficiently large $c>0$, we may pick $\tilde{\omega} = \frac{1}{m}(\iota^*\omega_{\rm FS} + c \cdot \pi^*\omega_{\mathbb{D}})$ as a K\"ahler form on $\mathcal{X}$, and $\tilde{\omega}_s = \tilde{\omega}|_{X_s} \in c_1(L_s)$. Let $\lambda(s) = |\log|s||$. For each $s\in \mathbb{D}^*$, let $\omega_s = \lambda(s)^{-1}\tilde{\omega}_s$, which is a K\"ahler form on $X_s$.

Let $e^{-\psi}$ be a Hermitian metric on $K_{X/\mathbb{D}^*}$ with analytic singularities on $\mathcal{X}$. 
As explained in \cite[Introduction]{BJ17}, $\nu_s = e^{2\psi_s}$ defines a measure on $X_s$ for $s\in\mathbb{D}^*$.
Consider the probability measure $\mu_s = \nu_s/\int_{X_s}\nu_s$ for $s\in\mathbb{D}^*$. By \cite[Theorem A, B]{BJ17}, $\mu_s$ converges weakly to a measure $\mu_0$ as probability measures on the hybrid space $X^{\rm hyb}$. 
In the following, we will consider the measure $\nu_s$ associated to a nowhere vanished holomorphic $n$-form $\Omega$ on $X$. See Section \ref{Uniform_Skoda} for the precise definition.
By the choice of our normalization, $\mu_0$ is supported on the essential skeleton of the Berkovich analytification $X^{\rm an}$.

Consider the family of parabolic Monge-Amp\`ere flows 
\begin{align}
    \label{family_of_parabolic_Monge_Ampere_flow}
    \begin{cases}
            e^{\partial_t\varphi_s}\;\mu_s = \frac{1}{V_{\omega_s}}(\omega_s+\ddc\varphi_s)^n, \\
            \varphi_s(0) = 0
        \end{cases}
\end{align}
defined on $X_s$, where $s\in \mathbb{D}^*$, $V_{\omega_s} = \int_{X_s}\omega_s^n$.
Write $\omega_{s,t} $ for $ \omega_s + \ddc\varphi_s(\cdot,t)$.

By \cite[Proposition 1.1]{Cao26}, there exist uniform constants $\alpha_0,A_0$, such that
\begin{align*}
    \int_{X_s} e^{-\alpha_0(u_s-\sup_{X_s}u_s)}\; \frac{\omega_{s}^n}{V_{\omega_s}} < A_0
\end{align*}
for any $u_s\in \PSH(X_s,\omega_{s})$.
Similarly,
by \cite[Theorem 1.3]{LiYang24}, there exist uniform constants $\alpha_\infty,A_\infty$, such that
\begin{align*}
    \int_{X_s} e^{-\alpha_\infty(u_s-\sup_{X_s}u_s)}\; \mu_s < A_\infty
\end{align*}
for any $u_s\in \PSH(X_s,\omega_{s})$. Let $\alpha = \min(\alpha_0,\alpha_\infty)$.

\begin{theorem}
    \label{main_thm_2}
    There exists $\varphi_s(x,t)\in C^\infty(X\times [0,+\infty))$ as a solution of the equation \eqref{family_of_parabolic_Monge_Ampere_flow}. There exists a uniform constant $C>0$ independent of $s\in \mathbb{D}^*$ and $t\in [0,+\infty)$, such that
    \begin{align}
        \label{main_thm_2_L_infty_bdd}
        \|\varphi_s\|_{L^\infty} \leq C
    \end{align}
    on $X_s\times [0,+\infty)$, and $\varphi_s(\cdot,t)$ converges to $\varphi_s(\cdot,\infty)\in C^\infty(X_s)$ smoothly, where $\varphi_s(\cdot,\infty)$ is the solution of
    \begin{align*}
        \frac{1}{V_{\omega_s}}(\omega_s+\ddc\varphi_s(\cdot,\infty))^n = \mu_s.
    \end{align*}
    
    In addition, there exists a uniform constant $A>0$ independent of $s\in \mathbb{D}^*$ and $t\in [0,+\infty]$, such that
    \begin{align}
        \label{main_thm_2_uniform_Skoda}
        \int_{X_s} e^{-\alpha(u_s-\sup_X u_s)}\; \frac{\omega_{s,t}^n}{V_{\omega_s}} < A
    \end{align}
    for any $u_s\in \PSH(X_s,\omega_{s,t})$.
\end{theorem}

At the end, we obtain  uniform estimates for some geometric properties of $(X_s,\omega_{s,t})$ by using the uniform estimates for Green function. In particular, this provides an alternative proof for the uniform diameter bound of Calabi-Yau metrics in \cite{LT2024}.

\begin{theorem}
    \label{main_thm_3}
    Let $\omega_{s,t}=\omega_s+\ddc\varphi_s(\cdot,t)$ be the family of K\"ahler metrics considered in Theorem \ref{main_thm_2}, where $s\in\mathbb{D}^*$, $t\in [0,+\infty]$. There exist uniform constants $C_d,C_v>0$, $\gamma>2n$ such that the following uniform diameter and local volume bounds hold on $X_s$:
    \begin{align}
        \label{main_thm_3_diameter_bound}
        {\rm diam}(\omega_{s,t}) &\leq C_d, \\
        \label{main_thm_3_volume_bound}
        \frac{1}{V_{\omega_{s,t}}}{\rm Vol}_{\omega_{s,t}}(B_{\omega_{s,t}}(x,r)) &\geq C_v r^\gamma
    \end{align}
    for any $x\in X_s$ and $0<r<\min(1,\frac{1}{3}{\rm diam}(\omega_{s,t}))$.
    
    In addition, fix each $t\in [0,+\infty]$, as $s\to 0$, up to picking a sub-sequence, $(X_s,\omega_{s,t})$ converges in Gromov-Hausdorff topology to a compact metric space $(Z_t,d_t)$.
\end{theorem}
\medskip
\begin{remark}
   The conclusions of Theorems \ref{main_thm_2} and \ref{main_thm_3} extend to the more general setting considered in \cite{BJ17}. Let $\pi:X\rightarrow\mathbb{D}^*$ be a projective meromorphic degeneration of polarized manifolds, equipped with a relatively ample $\mathbb{Q}$-line bundle $L$. Let $(\mathcal{X},\mathcal{L})$ be a simple normal crossing model of $(X,L)$, with $\mathcal{L}$ relatively ample. Let $e^{-\psi}$ be a smooth Hermitian metric on $K_{X/\mathbb{D}^*}$ with analytic singularities along the central fiber in the sense of \cite{BJ17}, and set
\begin{align*}
    \nu_s=e^{2\psi_s},
    \qquad
    \mu_s=\frac{\nu_s}{\int_{X_s}\nu_s}.
\end{align*}
Choose the background forms $\tilde{\omega}_s$ from a relative projective embedding as above, and set $\omega_s=\lambda(s)^{-1}\tilde{\omega}_s$.

By \cite[Proposition 1.1]{Cao26} and \cite[Theorem 8.2]{Bou2025}, the normalized volume measures $\omega_s^n/V_{\omega_s}$ and the measures $\mu_s$ satisfy uniform H\"ormander-Tian inequalities, respectively. Using local frames of the relative canonical bundle, the volume-ratio estimates in Section \ref{Uniform_Skoda} also extend to these measures. The arguments in Sections \ref{Uniform_Skoda} and \ref{Green's_function} therefore yield the conclusions of Theorems \ref{main_thm_2} and \ref{main_thm_3} for the parabolic Monge-Amp\`ere flow \eqref{family_of_parabolic_Monge_Ampere_flow}. In this setting, the corresponding metric evolution is a twisted K\"ahler--Ricci flow.

Since our main interest lies in applications to Calabi--Yau degenerations, we state the theorems in that setting.
\end{remark}

The argument of this paper has three stages. We first obtain uniform potential estimates for the parabolic  Monge-Amp\`ere equation. We then combine volume-ratio estimates with an interpolation argument to prove a uniform H\"ormander-Tian inequality along the flows. Finally, Green-function estimates yield the diameter and local volume bounds and Gromov--Hausdorff precompactness.
\medskip

\noindent {\bf Acknowledgments}.
This paper is written by a human being (the author). AI is used for proofreading and improving the exposition. 
The author thanks Prof. Yang Li for inspiring discussions. The author thanks Prof. Gang Tian for helpful suggestions. The author thanks Prof. Bin Guo for helpful communications.
The author is supported by  National Key R\&D Program of China 2024YFA1014800, NSFC-12671075 and XHD23A0101.
\medskip

\section{Estimates for the parabolic equation}
\label{estimates_for_the_parabolic_equation}
In this section, we show the proof of Theorem \ref{main_thm_1}. 
Since Yau's resolution of Calabi conjecture \cite{Yau78}, there has been a thorough study of the complex Monge-Amp\`ere equation, see \cite{Kolodziej98,EGZ09, CC21a,GPT2023,GP2022, WWZ20} for an incomplete list.
Our proof strategy is to use a parabolic version of apriori estimates in \cite{GPT2023}. We first use an auxiliary function modified from \cite{CC22} (see also \cite{Zhao25}) to prove a Moser-Trudinger type inequality and adapting an iteration scheme to prove the $L^\infty$-estimate.

\subsection{Moser-Trudinger type inequality}
\label{Moser-Trudinger_inequality}

Let $\psi_{s,t_0}\in \PSH(X,\omega)$ be the solution of the following parabolic equation
\begin{align}
    \label{auxiliary_function_1}
    \begin{cases}
        -\partial_t\psi_{s,t_0,t_1,j}\; \cdot \frac{1}{V_\omega} (\omega+\ddc\psi_{s,t_0,t_1,j})^n = f_{s,t_0,t_1,j} e^F\mu, \\
        \psi_{s,t_0,t_1,j}(x, t_0) = 0,
    \end{cases}
\end{align}
where $s, t_0, t_1>0$, and $a$ is any positive number (we only need to use $a=1$ for our purpose),
\begin{align*}
    t_1 &\geq t_0 + \sup_{x\in X}(|\varphi(x,t_0)|), \\
    \eta_j(x) &= \frac{x+\sqrt{x^2+j^{-1}}}{2}: \mathbb{R}\rightarrow \mathbb{R}_+, \\
     A_{s,t_0,t_1,j} &= \int_{t_0}^{t_1}\int_{X} \eta_j(-{\varphi}-s-t_1+t)^ae^F\mu\; dt, \\
     f_{s,t_0,t_1,j} &= \frac{1}{A_{s,t_0,t_1,j}} \eta_j(-{\varphi}-s-t_1+t)^a. \\
     \Omega_{s,t_0,t_1} &= \{(x,t)\in X\times[t_0,t_1]: -{\varphi}(x,t)\geq s+t_1-t\}.
\end{align*}
The functino $\eta_j$ is smooth and positive, satisfies $\eta_{j+1}< \eta_j$, $\eta_j(x)>x$ and $\lim_{j\to\infty}\eta_j = \max(x,0)$. 
By \cite[Proposition 2.3]{CC22}, equation \eqref{auxiliary_function_1} admits a smooth solution $\psi_{s,t_0,t_1,j}$ with $L^\infty$-estimate independent of $j$ on any finite time interval $[0,T]$. 
Moreover, $A_{s,t_0,t_1,j}$ converges decreasingly to a limit $A_{s,t_0,t_1}$.


\begin{lemma}
    \label{auxiliary_function_lemma}
    Let $\psi_{s,t_0,t_1,j}$ be the solution of the parabolic equation \eqref{auxiliary_function_1}.
    Then there exists $C_1 = C_1(n,\alpha,A,K)>0$ such that
    \begin{align*}
        \sup_{t\in [t_0,t_1]} \int_X e^{-\alpha\cdot \psi_{s,t_0,t_1,j}}\mu \leq C_1.
    \end{align*}
\end{lemma}
\begin{proof}
    The Monge-Amp\`ere energy of $\psi_{s,t_0,t_1,j}$ is 
    \begin{align*}
         E(\psi_{s,t_0,t_1,j}) = \frac{1}{(n+1)V_\omega}\sum_{0\leq k\leq n}\int_X\psi_{s,t_0,t_1,j}\; \omega^k\wedge\omega_{\psi_{s,t_0,t_1,j}}^{n-k}.
    \end{align*}
    Differentiating the Monge-Amp\`ere energy along the flow gives
    \begin{align*}
        \frac{d}{dt}E(\psi_{s,t_0,t_1,j}) &= \frac{1}{V_\omega}\int_X \partial_t \psi_{s,t_0,t_1,j}(\omega+\ddc\psi_{s,t_0,t_1,j})^n \\
        &= \int_{X} -f_{s,t_0,t_1,j} e^F\mu \leq 0.
    \end{align*}
    Integrating in time, we obtain 
    \begin{align*}
       E(\psi_{s,t_0,t_1,j}(t)) &= E(\psi_{s,t_0,t_1,j}(t))-E(\psi_{s,t_0,t_1,j}(t_0)) = \int_{t_0}^{t}\int_{X} -f_{s,t_0,t_1,j}e^F\mu \; d t \\
        &\geq \frac{-1}{A_{s,t_0,t_1,j}}\int_{t_0}^{t_1}\int_{X} \eta_j(-{\varphi}-s-t_1+t)^a e^F\mu \; dt = -1.
    \end{align*}
    The $J$ function plays the role of a distance in the study of $\mathcal{E}^1(X,\omega)/\mathbb{R}^*$ space, see \cite{Darvas15} and also \cite{BBEGZ19} for details. In particular, 
    $J(\psi_{s,t_0,t_1,j}) = \frac{1}{V_\omega}\int_X\psi_{s,t_0,t_1,j}\omega^n-E(\psi_{s,t_0,t_1,j}) \geq 0$. Thus, for any $t\in [t_0,t_1]$, we have
    \begin{align*}
        \sup_X \psi_{s,t_0,t_1,j} \geq \frac{1}{V_\omega}\int_X\psi_{s,t_0,t_1,j}\omega^n \geq  -1.
    \end{align*}
  Applying the uniform H\"ormander-Tian inequality now gives
    \begin{align*}
        \int_Xe^{-\alpha\psi_{s,t_0,t_1,j}}\; \mu &= \int_X e^{-\alpha(\psi_{s,t_0,t_1,j}-\sup_X\psi_{s,t_0,t_1,j})}\mu\cdot e^{-\alpha\sup_X\psi_{s,t_0,t_1,j}} \\
        &\leq A\cdot e^{\alpha} =: C_1.
    \end{align*}
\end{proof}

\begin{lemma}
    \label{auxiliary_inequality_varphi_psi}
    There exists a uniform constant $C_2 = C_2(n,\alpha,A,K)>0$ such that, for every $t\in [t_0,t_1]$
    \begin{align}
    \label{ineqaultiy_a}
    \frac{(-\varphi-s-t_1+t)_+}{A_{s,t_0,t_1,j}^{\frac{1}{n+1+a}}} \leq \Big( C_2 \big(-\psi_{s,t_0,t_1,j}+C_2 A_{s,t_0,t_1,j}^{\frac{1}{a}}\big)\Big)^{\frac{n+1}{n+1+a}}.
\end{align}    
\end{lemma}
\begin{proof}
    
By equation \eqref{auxiliary_function_1}, $\partial_t\psi_{s,t_0,t_1,j}\leq 0$. Since $\psi_{s,t_0,t_1,j}(\cdot,t_0)=0$, $\psi_{s,t_0,t_1,j}(\cdot,t)\leq 0$ for $t\in [t_0,t_1]$.

Let $L = -\partial_t+\Delta_\varphi$ and define
\begin{align*}
    H = -\epsilon(-\psi_{s,t_0,t_1,j}+\Lambda)^b-\varphi-s-t_1+t,
\end{align*}
where
\begin{align*}
    b = \frac{n+1}{n+1+a}, \epsilon = b^{-b}A_{s,t_0,t_1,j}^\frac{1}{n+1+a}, \Lambda = b A_{s,t_0,t_1,j}^{\frac{1}{a}}.
\end{align*}
These choices satisfy $\epsilon b\Lambda^{b-1}=1$.

Our goal is to prove that $H\leq 0$ on $X\times [t_0,t_1]$. Assume the contrary. Then there exists $(x',t')\in X\times [t_0,t_1]$ such that $H(x',t') = \sup_{(x,t)\in X\times[t_0,t_1]}H(x,t) > 0$. Since $-\varphi \leq t_1-t_0$ by the definition, $\sup_{x\in X}H(x,t_0)\leq 0$. Then $(x',t')\in X\times (0,t_1]$, and $L(H)(x',t')\leq 0$ by parabolic maximum principle.
At $(x',t')$, we obtain
\begin{align*}
    0 &\geq L(H) \\
    &\geq -\epsilon\cdot b (-\psi_{s,t_0,t_1,j}+\Lambda)^{b-1}\partial_t\psi_{s,t_0,t_1,j} + \partial_t\varphi+\epsilon\cdot b (-\psi_{s,t_0,t_1j}+\Lambda)^{b-1}\tr_\varphi\omega_{\psi_{s,t_0,t_1,j}} - n -1 \\
    & \;\; - \epsilon\cdot b(-\psi_{s,t_0,t_1,j}+\Lambda)^{b-1}\tr_\varphi\omega + \tr_\varphi\omega \\
    &\geq \epsilon\cdot b(-\psi_{s,t_0,t_1,j}+\Lambda)^{b-1}(-\partial_t\psi_{s,t_0,t_1,j}+\tr_\varphi\omega_{\psi_{s,t_0,t_1,j}})+\partial_t\varphi-n-1,
\end{align*}
where in the last inequality we used $1=\epsilon b\Lambda^{b-1}\geq \epsilon\cdot b(-\psi_{s,t_0,t_1,j}+\Lambda)^{b-1}$.
Meanwhile,
\begin{align*}
    -\partial_t\psi_{s,t_0,t_1,j}&+\tr_\varphi\omega_{\psi_{s,t_0,t_1,j}} \geq -\partial_t\psi_{s,t_0,t_1,j} + n\big( \frac{\omega_{\psi_{s,t_0,t_1,j}}^n}{\omega_\varphi^n}\big)^\frac{1}{n} \\
    &\geq -\partial_t\psi_{s,t_0,t_1,j}+n\big(\frac{f_{s,t_0,t_1,j}}{-\partial_t\psi_{s,t_0,t_1,j}} e^{-\partial_t\varphi}\big)^{\frac{1}{n}} \\
    &\geq (n+1)f_{s,t_0,t_1,j}^{\frac{1}{n+1}}\cdot \exp\big(\frac{-1}{n+1}\partial_t\varphi),
\end{align*}
where in the last line we use Young's inequality.
Combining these inequalities, we obtain
\begin{align}
    \label{LH}
    &0 \geq LH \\ \nonumber
    &\geq (n+1)\epsilon\cdot b(-\psi_{s,t_0,t_1,j} +\Lambda)^{b-1}f_{s,t_0,t_1,j}^{\frac{1}{n+1}} \cdot \exp\big(\frac{-1}{n+1}\partial_t\varphi\big) +\partial_t\varphi-n-1.
\end{align}
Since $H(x',t')>0$, we have $(-\psi_{s,t_0,t_1,j} +\Lambda)^{b} < \epsilon^{-1}(-\varphi-s-t_1+t)$ at $(x',t')$. Hence
\begin{align*}
    (-\psi_{s,t_0,t_1,j} +\Lambda)^{b-1} > \big( \epsilon^{-1} (-\varphi-s-t_1+t) \big)^{\frac{b-1}{b}}.
\end{align*}
Consequently,
\begin{align*}
    (n+1) &\epsilon \cdot b(-\psi_{s,t_0,t_1,j} +\Lambda)^{b-1}f_{s,t_0,t_1,j}^{\frac{1}{n+1}} \\
    &> (n+1)\epsilon^{\frac{1}{b}}b(-\varphi-s-t_1+t)^{\frac{-a}{n+1}} \frac{(-\varphi-s-t_1+t)^{\frac{a}{n+1}}}{A_{s,t_0,t_1,j}^{\frac{1}{n+1}}} = n+1.
\end{align*}
Substituting this bound into \eqref{LH}, we obtain
\begin{align*}
    0 >  (n+1)e^{\frac{-1}{n+1}\partial_t\varphi}+\partial_t\varphi-(n+1).
\end{align*}
However, $(n+1)e^{\frac{-1}{n+1}x}+x-(n+1)$ is a non-negative function. This leads to a contradiction. Therefore,
$H(x,t)\leq 0$ on $X\times [t_0,t_1]$.
Rewriting this inequality gives \eqref{ineqaultiy_a}.
\end{proof}

\vspace{.15in}

\begin{lemma}
    \label{sup_X_varphi}
    There exists a uniform constant $C_3 = C_3(n,\alpha,A,K)>0$, such that
    \begin{align}
    \label{sup_bound_average_bound}
        0 \leq  \sup_X \varphi \leq C_3, \\
        -\int_X  \varphi \;  \mu\leq C_3. \nonumber
    \end{align}
\end{lemma}
\begin{proof}

Recall that
\begin{align*}
    E(\varphi) = \frac{1}{(n+1)V_\omega}\sum_{0\leq k\leq n}\int_X\varphi\; \omega^k\wedge\omega_\varphi^{n-k},
\end{align*}
and 
\begin{align*}
    \frac{d}{dt}E = \frac{1}{V_\omega}\int_X\partial_t\varphi\;\omega_\varphi^n = \int_X\log\big(\frac{\frac{1}{V_\omega}\omega_\varphi^n}{e^F\mu}\big)\frac{\omega_\varphi^n }{V_\omega} \geq 0.
\end{align*}
At $t=0$, we have $\varphi = 0$ and $E(\varphi) = 0$. Therefore, 
\begin{align*}
    \sup_X\varphi \geq E(\varphi)\geq 0.
\end{align*}
On the other hand,
by equation \eqref{parablic_MA},
\begin{align*}
    \frac{d}{dt}\int_X \varphi\; e^F \mu &= \int_X\log\big( \frac{\frac{1}{V_\omega}\omega_\varphi^n}{e^F \mu} \big) \; e^F\mu \leq 0.
\end{align*}
It follows that $\int_X\varphi \; e^F\mu \leq 0$.
Indeed, the preceding monotonicity and the initial value of $\varphi$ give $\int_X \varphi \; e^F\mu \leq 0$.
Moreover, the generalized Young's inequality used in \cite[Lemma 1]{GPT2023} gives
\begin{align*}
    \frac{\alpha}{2}(\sup_X\varphi-\varphi)\cdot e^F &\leq e^F(1+|F|)+ \frac{\alpha}{2}(\sup_X\varphi-\varphi) \cdot e^{\frac{\alpha}{2}(\sup_X\varphi -\varphi)} \\
    &\leq e^F(1+|F|)+e^{\alpha(\sup_X\varphi-\varphi)}.
\end{align*}
Integrating, we obtain
\begin{align*}
    \frac{\alpha}{2}\int_X (\sup_X\varphi - \varphi) e^F\mu &\leq \int_Xe^F(1+|F|)\;\mu+\int_X e^{-\alpha(\varphi-\sup_X\varphi)}\mu \\
    &\leq 1+K^{\frac{1}{p}} +A.
\end{align*}
Therefore,
\begin{align*}
    \sup_X\varphi \leq \frac{2}{\alpha}(1+K^{\frac{1}{p}}+A),
\end{align*}
which is uniformly bounded from above.

The uniform H\"ormander-Tian inequality \eqref{uniform_Skoda} also yields
\begin{align*}
    \int_X e^{-\alpha\varphi}\;\mu \leq A\cdot e^{-\alpha\sup_X\varphi} \leq A.
\end{align*}
Jensen's inequality therefore gives
\begin{align}
    \label{average_bound_omega}
    -\int_X \varphi \; \mu \leq \frac{1}{\alpha}\log\Big(\int_X e^{-\alpha\varphi}\;\mu\Big) \leq \frac{1}{\alpha}\log(A).
\end{align}
Taking $C_3= \max(\frac{2}{\alpha}(1+K^{\frac{1}{p}}+A), \frac{1}{\alpha}\log(A))$, we obtain the required inequality \eqref{sup_bound_average_bound}.
\end{proof}

For any $T>1$, define
\begin{align}
     \bar{E}(\varphi) = \sup_{t_0\in[0,T-1]}\int_{t_0}^{t_0+1}\int_{X}(\sup_X\varphi-\varphi) e^F \mu \; dt.
\end{align}

\begin{lemma}
    \label{Morrey-Trusinger}
    Let $T>1$ and $t_0\in [0,T)$. Assume that there exists a unit interval $I$ such that
    \begin{align*}
        \Omega_{s,t_0,t_1} \subset X\times (I\cap [t_0,t_1])\subset [0,T].
    \end{align*}
    Then there exists a uniform constant $C_4 = C_4(n,\alpha,A,K)>0$ such that
    \begin{align*}
        \int_{\Omega_{s,t_0,t_1}} \exp\Big(\beta\big(\frac{(-\varphi-s-t_1+t)_+}{A_{s,t_0,t_1,j}^{\frac{1}{n+2}}}\big)^{\frac{n+2}{n+1}}\Big) \;\mu \; dt\leq C_4\cdot \exp\Big(C_4\big(\bar{E}(\varphi)+j^{\frac{-1}{2}}(t_1-t_0)\big)\Big),
    \end{align*}
    where $\beta =  C_2^{-1}\cdot \alpha$.
   
\end{lemma}
\begin{proof}
    For $a=1$,  inequality \eqref{ineqaultiy_a} and Lemma \ref{auxiliary_function_lemma} gives 
    \begin{align*}
        &\int_{I\cap[t_0,t_1]}\int_X \exp\Big(\beta\big(\frac{(-\varphi-s-t_1+t)_+}{A_{s,t_0,t_1,j}^{\frac{1}{n+2}}}\big)^{\frac{n+2}{n+1}}\Big) \;\mu \; dt  \\
        &\leq \int_{I\cap [t_0,t_1]} \int_X \exp\Big(\alpha(-\psi_{s,t_0,t_1,j}+C_2 A_{s,t_0,t_1,j}\big)\Big)\;d\mu\;dt 
        \leq C_1\cdot \exp\big(C_2 A_{s,t_0,t_1,j}\big).
    \end{align*}
    Moreover,
    \begin{align*}
        A_{s,t_0,t_1,j} &= \int_{t_0}^{t_1}\int_{X}\eta_j(-\varphi-s-t_1+t) e^F\mu\;dt \\
        &\leq\int_{\Omega_{s,t_0,t_1}} (-\varphi+C_3)e^F\mu\;dt + \int_{\Omega_{s,t_0,t_1}}(-C_3-s-t_1+t)e^F\mu\;dt +\frac{1}{2}j^{\frac{-1}{2}} \\
        &\leq\int_{X\times I} (-\varphi+C_3) e^F\mu \;dt + \int_{X\times I\setminus\Omega_{s,t_0,t_1}}(\varphi-C_3)e^F\mu\;dt +\frac{1}{2}j^{\frac{-1}{2}}(t_1-t_0) \\
        &\leq \bar{E}(\varphi)+C_3 +\frac{1}{2}j^{\frac{-1}{2}}(t_1-t_0),
    \end{align*}
    where we use Lemma \ref{sup_X_varphi} for the last inequality.
    Thus there exists a uniform constant $C_4>0$ such that
    \begin{align*}
          \int_{\Omega_{s,t_0,t_1}}\exp\Big(\beta\big(\frac{(-\varphi-s-t_1+t)_+}{A_{s,t_0,t_1,j}^{\frac{1}{n+2}}}\big)^{\frac{n+2}{n+1}}\Big) \;\mu \; dt\leq C_4\cdot \exp\Big(C_4\big(\bar{E}(\varphi)+j^{\frac{-1}{2}}(t_1-t_0)\big)\Big).
    \end{align*}
\end{proof}

We next obtain a uniform bound for the energy quantity $\bar{E}$.
\begin{lemma}
    \label{energy_bounded_by_p_entropy}
    For any $T>1$,
    there exists a uniform constant $C_5 = C_5(n,\alpha,A,K)>0$ such that for any $1\leq q \leq p$,
    \begin{align}
        \label{q_moment_bounded_by_p_entropy}
        \sup_{t_0\in [0,T-1]}\int_{t_0}^{t_0+1}\int_X (\sup_X\varphi-\varphi)^q e^F\mu \; dt \leq C_5.
    \end{align}
    In particular,
    \begin{align}
        \label{energy_bounded_by_p_entropy}
        \bar{E}(\varphi)\leq C_5.
    \end{align}
\end{lemma}
\begin{proof}
   Applying the generalized Young's inequality again, we obtain
    \begin{align*}
        \big(\frac{\alpha}{2}(\sup_X\varphi-\varphi)\big)^p e^F \leq e^F(1+|F|)^p + C(p)e^{\alpha(\sup_X\varphi-\varphi)}.
    \end{align*}
    Therefore,
    \begin{align*}
        \int_{t_0}^{t_0+1}\int_X (\sup_X\varphi-\varphi)^p e^F\mu \; dt &\leq C_{5,1}\int_{t_0}^{t_0+1}\Big(\int_X e^F(1+|F|)^p\mu +e^{-\alpha(\varphi-\sup_X\varphi)} \mu \Big) dt \\
        &\leq C_{5,2}.
    \end{align*}
    H\"older's inequality then gives \eqref{q_moment_bounded_by_p_entropy} for $1\leq q\leq p$.
\end{proof}

\begin{proposition}
    \label{L^infty_bound_by_energy}
    There exists a constant $C_6 = C_6(n,\alpha,A,K,\bar{E}(\varphi))>0$, such that for any $T>1$,
    \begin{align}
        \label{L_infty_bound_1}
        \|\varphi\|_{L^\infty(X\times [0,T])} \leq C_6.
    \end{align}
\end{proposition}
\begin{proof}
Without loss of generality, assume that $M:=\sup_{(x,t)\in X\times [0,T]}(-\varphi(x,t)) > C_3+1$.
Set $(x_1,t_1)$ be the point in $X\times (0,T]$ such that $-\varphi(x_1,t_1) = \sup_{(x,t)\in X\times [0,T]}(-\varphi(x,t))$.
Let $t_0 = \max(0,t_1-M)$. Then $t_1-t_0 \geq \sup_{x\in X}|\varphi(x,t_0)|$. Let $s_0 = M-1$. The set
$\Omega_{s_0,t_0}$ is nonempty, and any $(x,t)\in \Omega_{s_0,t_0,t_1}$ satisfies
\begin{align*}
   0\leq t_1-t \leq -\varphi(x,t)-s_0 \leq M-(M-1)=1.
\end{align*}
It follows that $\Omega_{s_0,t_0,t_1} \subset X\times [t_1-1,t_1]$. 
Set $s\geq s_0$. Then the condition in the statement of Lemma \ref{Morrey-Trusinger} is satisfied.

Applying the generalized Young's inequality, we obtain
\begin{align*}
    \Big(\frac{\beta}{2}\big(\frac{(-\varphi-s-t_1+t)_+}{A_{s,t_0,t_1,j}^{\frac{1}{n+2}}}\big)^{\frac{n+2}{n+1}}\Big)^p &\cdot e^F \leq \\
    & e^F(1+|F|)^p+ C(p)\cdot \exp\Big(\beta\big(\frac{(-\varphi-s-t_1+t)_+}{A_{s,t_0,t_1,j}^{\frac{1}{n+2}}}\big)^{\frac{n+2}{n+1}}\Big).
\end{align*}
Then by Lemma \ref{Morrey-Trusinger},
\begin{align*}
    &\int_{\Omega_{s,t_0,t_1}}  \Big(\frac{\beta}{2}\big(\frac{(-\varphi-s-t_1+t)_+}{A_{s,t_0,t_1,j}^{\frac{1}{n+2}}}\big)^{\frac{n+2}{n+1}}\Big)^p\cdot e^F \mu\;dt \\
    &\leq \int_{\Omega_{s,t_0,t_1}}(1+|F|)^p e^F \mu\;dt + C(p)\int_{\Omega_{s,t_0,t_1}} \exp\Big(\beta\big(\frac{(-\varphi-s-t_1+t)_+}{A_{s,t_0,t_1,j}^{\frac{1}{n+2}}}\big)^{\frac{n+2}{n+1}}\Big) \mu \;dt \\
    &\leq 2^p(1+K)+C_p C_4\cdot e^{C_4\bar{E}(\varphi)},
\end{align*}
which implies that
\begin{align*}
    \int_{\Omega_{s,t_0,t_1}}(-\varphi-s-t_1+t)_+^{\frac{(n+2)p}{n+1}} e^F\mu\;dt \leq \beta^{-p}\cdot A_{s,t_0,t_1,j}^{\frac{p}{n+1}}\cdot C_{6,1}(1+e^{C_4(\bar{E}(\varphi)+j^{\frac{-1}{2}}(t_1-t_0))}).
\end{align*}
Define the measure of the sublevel set by
\begin{align*}
     \theta_{t_0,t_1}(s) &:= \int_{\Omega_{s,t_0,t_1}}e^F \mu\; dt.
\end{align*}
Let $I = [\max(t_0,t_1-1),t_1]$.
By the definition of $\Omega_{s,t_0,t_1}$ and the inequality $s\geq M-1$, 
    \begin{align}
    \label{sublevel_bound}
        \theta_{t_0,t_1}(s) &\leq \int_{\Omega_{s,t_0,t_1}} \frac{-\varphi}{t_1-t+s}\; e^F\mu \; dt\\ \nonumber
        &\leq \frac{1}{M-1}\int_{I}\int_X (-\varphi)_+e^F\mu\;dt \\ \nonumber
        &\leq \frac{1}{M-1}\int_{\max(0,t_1-1)}^{t_1}\int_X (\sup_X\varphi-\varphi) e^F\mu \; dt \leq \frac{\bar{E}(\varphi)}{M-1}.
    \end{align}

Set $A_{s,t_0,t_1} = \lim_{j\to\infty} A_{s,t_0,t_1,j}$. H\"older's inequality gives
    \begin{align*}
        A_{s,t_0,t_1} &= \int_{\Omega_{s,t_0,t_1}}(-\varphi-s-t_1+t)_+e^F\mu\;dt \\
        &\leq \Big(\int_{\Omega_{s,t_0,t_1}}(-\varphi-s-t_1+t)_+^{\frac{(n+2)p}{n+1}}e^F\mu\;dt\Big)^{\frac{n+1}{(n+2)p}}\cdot \theta_{t_0,t_1}(s)^{1-\frac{n+1}{(n+2)p}} \\
        &\leq \big(C_{6,1}(1+e^{C_4\bar{E}(\varphi)})\big)^{\frac{n+1}{(n+2)p}}\cdot\beta^{-\frac{n+1}{n+2}}\cdot A_{s,t_0,t_1}^{\frac{1}{n+2}}\cdot \theta_{t_0,t_1}(s)^{1-\frac{n+1}{(n+2)p}},
    \end{align*}
   Set $C_{6,2} = \big(C_{6,1}(1+e^{C_4\bar{E}(\varphi)})\big)^{\frac{1}{p}}\cdot\beta^{-1}$, $1+\delta = \frac{n+2}{n+1}-\frac{1}{p}$, then we have
   \begin{align*}
       A_{s,t_0,t_1} \leq C_{6,2}\cdot \theta_{t_0,t_1}(s)^{1+\delta}.
   \end{align*}
   On the other hand,
   \begin{align*}
       A_{s,t_0,t_1} \geq \int_{\Omega_{s+r,t_0,t_1}}(-\varphi-s-t_1+t)_+e^F\mu\;dt \geq r\cdot \theta_{t_0,t_1}(s+r).
   \end{align*}
   Combining these bounds gives the iteration inequality
   \begin{align}
       r\cdot \theta_{t_0,t_1}(s+r)\leq C_{6,2}\cdot \theta_{t_0,t_1}(s)^{1+\delta}
   \end{align}
   for any $r>0$, $s\geq s_0$. 
    Next we apply De Giorgi's iteration scheme.
    Let $s_{i+1} = s_i+r_i$ for $i\in\mathbb{N}$, $r_i = eC_{6,2}\theta_{t_0,t_1}(s_i)^\delta$. Then $s_{i+1} = s_0+\sum_{0\leq j\leq i}r_j$, and
    \begin{align*}
        \log\theta_{t_0,t_1}(s_{i+1}) -\log\theta_{t_0,t_1}(s_i) &\leq -\log(r_i)+\log(C_{6,2})+\delta\log\theta_{t_0,t_1}(s_i) =-1.
    \end{align*}
    Since $\theta_{t_0,t_1}(s)$ is a decreasing continuous function, the iteration above shows that $s_\infty \geq \sup_{X\times [0,T]}(-\varphi) = M$. Then 
    \begin{align*}
      M\leq   s_\infty &= s_0+ \sum_{i=0}r_i \leq s_0+C_{6,2}e\big(\sum_{i=0}^\infty e^{-\delta i}\big)\theta_{t_0,t_1} (s_0)^\delta \\
        &\leq M-1+ \frac{eC_{6,2}}{1-e^{-\delta}}\theta_{t_0,t_1}(s_0)^\delta.
    \end{align*}
     Combine inequality \eqref{sublevel_bound} with the inequality above, we have
     \begin{align*}
         \big(\frac{eC_{6,2}}{1-e^{-\delta}}\big)^{-\frac{1}{\delta}} \leq \frac{\bar{E}(\varphi)}{M-1}
     \end{align*}
    Taking $C_6 = 1+ \big(\frac{eC_{6,2}}{1-e^{-\delta}}\big)^{\frac{1}{\delta}}\cdot \bar{E}(\varphi)$,
   we obtain the inequality \eqref{L_infty_bound_1}.
\end{proof}

\subsection{Uniform estimates and smooth convergence}
\label{conclusion_thm_1}
We now complete the proof of Theorem \ref{main_thm_1}. Lemma \ref{energy_bounded_by_p_entropy} gives a uniform upper bound for $\bar{E}(\varphi)$. Thus the bound $C_6$ for $\|\varphi\|_{L^\infty(X\times [0,T))}$ in Proposition \ref{L^infty_bound_by_energy} is uniform. This proves the uniform $L^\infty$-estimate for $\varphi$ asserted in Theorem \ref{main_thm_1}.
\medskip

The strategy to show the smooth convergence of the flow is quite standard, so we only sketch the key steps of the proof.  The constants $\hat{C}_i$ below are uniform in time $t$, but may depend on the given flow. Since
\begin{align*}
    (\partial_t-\Delta_{\omega_\varphi})\partial_t\varphi = 0,
\end{align*}
both the supremum and infimum of $\partial_t\varphi$ are attained at $t=0$. It follows that 
\begin{align*}
    \|\partial_t\varphi\|_{L^\infty(X\times [0,+\infty))} \leq \hat{C}_1,
\end{align*} 
and also the boundedness of the volume ratio. Then we can choose a constant $\hat{C}_2>0$ such that
\begin{align*}
    (\partial_t-\Delta_{\omega_\varphi})(\log(n+\Delta\varphi)-\hat{C}_2\varphi) \leq -\hat{C}_2\tr_{\omega_\varphi}\omega + \hat{C}_3.
\end{align*}
At a maximum point of $\log(n+\Delta\varphi)-\hat{C}_2\varphi$, we obtain that $\tr_{\omega_\varphi}\omega\leq \frac{\hat{C}_3}{\hat{C}_2}$. Then by the volume ratio bound and the $L^\infty$-bound of $\varphi$, we have $\|\varphi\|_{C^{1,\bar{1}}(X)}\leq \hat{C}_4$ for any $t\in [0,+\infty)$. Therefore, the flow is uniformly parabolic, and standard estimates give uniform higher-order bounds for $t\in [0,+\infty)$.

To prove convergence, we consider the twisted Mabuchi functional
\begin{align*}
    M(\varphi) = H(\varphi)-E^{Ric(\omega)-\chi}(\varphi),
\end{align*}
where $\chi = -\partial\bar\partial\log(e^F\mu)$ and
\begin{align*}
    H(\varphi) &= \int_{X}\log\big(\frac{\omega_\varphi^n}{\omega^n}\big)\frac{\omega_\varphi^n}{V_\omega}, \\
     E^{Ric(\omega)-\chi} &= \frac{1}{V_\omega}\sum_{i=0}^{n-1}\int_X \varphi(Ric(\omega)-\chi)\wedge\omega^i\wedge\omega_\varphi^{n-1-i}.
\end{align*}
A direct computation gives,
\begin{align*}
    \frac{d}{dt}M(\varphi) &= \frac{1}{V_\omega}\int_X\partial_t\varphi\;(-\Delta_{\omega_\varphi}\log(e^F\mu)-R(\omega_\varphi))\;\omega_\varphi^n \\
    &=\frac{1}{V_\omega}\int_X \partial_t\varphi \cdot \Delta_{\omega_\varphi}\log(\frac{\omega_\varphi^n}{e^F\mu}) \;\omega_\varphi^n \\
    &= -\frac{1}{V_\omega}\int_X |\nabla\partial_t\varphi|^2_{\omega_\varphi}\;\omega_\varphi^n \leq 0.
\end{align*}
Again by direct computation,
\begin{align*}
    \frac{d^2}{dt^2}M(\varphi) \geq \hat{C}_5 \frac{d}{dt}M(\varphi).
\end{align*}
Since $M(\varphi)$ is bounded below, by the computations above, for any sequence $t_i\rightarrow \infty$, $\frac{d}{dt}M(\varphi)$ converges to $0$. Together with the uniformly higher-order estimates, this shows that $\omega_\varphi$ converges to a limit $\omega_\infty$ satisfying $\frac{1}{V_\omega}\omega_\infty^n = e^F\mu$. By the uniqueness of the solution, $\omega_\infty$ is unique, and the determination of the potential function is up to the difference of a constant. Finally, since 
\begin{align*}
    \int_X\varphi \; e^F\mu
\end{align*}
is decreasing along the flow, this implies that $\varphi$ converges smoothly to $\varphi_\infty$, such that $\omega_{\varphi_\infty} = \omega_\infty$. This completes the proof.

\section{Uniform H\"ormander-Tian inequality along the flow}
\label{Uniform_Skoda}

In this section, we show the proof of Theorem \ref{main_thm_2}. We will follow the notions introduced in the Introduction section. Our strategy is to apply Theorem \ref{main_thm_1} to obtain uniform apriori estimates. Then we adapt an interpolation argument, which reduces the proof of uniform H\"ormander-Tian inequality to estimates of volume ratios and a uniform H\"ormander-Tian inequality with respect to Fubini-Study metrics (\cite[Proposition 2.5]{DNGG23}).
\medskip

As stated in the Introduction, by \cite[Proposition 1.1]{Cao26} and \cite[Theorem 1.3]{LiYang24}, the conditions in the statement of Theorem \ref{main_thm_1} are satisfied for any $s\in\mathbb{D}^*$. The uniform $L^\infty$-estimate and smooth convergence therefore follows from Theorem \ref{main_thm_1}.
\medskip

In the following, we will prove the uniform H\"ormander-Tian inequality. We begin by comparing the measures $\tilde{\omega}_s^n$ and $\mu_s$.
Since $e^{-\psi}$ has analytic singularities, there exists a $\mathbb{Q}$-line bundle $\mathcal{K}$ on $\mathcal{X}$, such that $\mathcal{K}|_X = K_{X/\mathbb{D}^*}$, and $e^{-\psi}$ extends to a Hermitian metric on $\mathcal{K}$. The Hermitian metric $e^{-\psi}$ equivalently defines a measure $\nu_s$ on $X_s$. Let $\{E_i\}_{i\in I}$ be the set of irreducible components of the central fiber of $\mathcal{X}$. Since $\mathcal{X}$ is a semi-stable reduction, $\mathcal{X}_0 = \sum_{i\in I}E_i$ is a simple normal cross divisor. The relative log canonical divisor is given by
\begin{align*}
    K_{\mathcal{X}/\mathbb{D}}^{\log} := (K_{\mathcal{X}}+\mathcal{X}_{0,{\rm red}})-\pi^*(K_{\mathbb{D}}+[0]) = K_{\mathcal{X/\mathbb{D}}}.
\end{align*}
We write
\begin{align*}
    K_{\mathcal{X}/\mathbb{D}}^{\log} = \mathcal{K} + \sum_{i\in I} a_i E_i,
\end{align*}
where the coefficient $a_i$ is the log discrepancy along $E_i$. We use the normalizaton $\min_{i\in I}(a_i) = 0$. The skeleton $\Delta_{\mathcal{K}}\subset \Delta_{\mathcal{X}}$ is the union of the valuations where the log discrepancy equal to zero.
For $J\subset I$, set $E_J = \cap_{i\in J}E_i$.
Let $x\in E_J$ and let
$\mathcal{U}$ be an open neighborhood of $x$ in $\mathcal{X}$. Let $(z_0,\ldots,z_n)$ be a coordinate of $\mathcal{U}$. We may assume that $|J|=p+1$ and that on $\mathcal{U}$, $s= \prod_{0\leq i\leq p} z_i$.

For the Calabi-Yau family, choose $\Omega$ as a global section of $\mathcal{K}$ whose restriction on $X/\mathbb{D}^*$ is a nowhere-vanishing holomorphic $n$-form. We consider the measure $\nu_s = (\sqrt{-1})^{n^2}\Omega\wedge\overline{\Omega}|_{X_s}$.

 \begin{lemma}
     \label{volume_ratio_bound}
     There exist uniform constant $C_1, C_2>0$ independent of $s\in \mathbb{D}^*$, such that
     \begin{align}
        \label{volume_ratio_bound_inequality}
         C_1^{-1}\cdot |s|^{2C_2} \leq \frac{\tilde{\omega}_s^n}{\mu_s} \leq C_1\cdot |s|^{-2 C_2}.
     \end{align}
 \end{lemma}
 \begin{proof}
Let $\mathcal{U}$ be a precompact open chart where $z_0\neq 0$.
On the local chart $\mathcal{U}$, the line bundle $K_{\mathcal{X}}$ has a local trivialization $\Omega_{\mathcal{X}} = dz_0\wedge\ldots\wedge dz_n$.
By the adjunction formula, on $X_s\cap \mathcal{U}$, we have a local trivialization of $K_{\mathcal{X}/\mathbb{D}}$:
\begin{align*}
    \Omega_{\mathcal{X}_s} =  \big(\bigwedge_{1\leq i\leq p}\frac{dz_i}{z_i}\big)\wedge\big(\bigwedge_{p+1\leq j\leq n} dz_j\big).
\end{align*}
On $X_s\cap \mathcal{U}$, we have
\begin{align*}
    \tilde{\omega}_s &\approx \sqrt{-1}\sum_{0\leq i\leq n} dz_i\wedge d\bar{z}_i \\
    &= \sqrt{-1}\sum_{1\leq i\leq n} dz_i\wedge d\bar{z}_i + 
    \sqrt{-1}\sum_{1\leq i,j\leq p} \frac{|z_0|^2}{z_i\bar{z}_j}dz_i\wedge d\bar{z}_j.
\end{align*}
Since
\begin{align*}
    &\bigwedge^n \big(\sqrt{-1}\sum_{1\leq i\leq n} dz_i\wedge d\bar{z}_i + 
    \sqrt{-1}\sum_{1\leq i,j\leq p} \frac{|z_0|^2}{z_i\bar{z}_j}dz_i\wedge d\bar{z}_j\big) \\
    &= \big(1+\sum_{1\leq i\leq p}\frac{|z_0|^2}{|z_i|^2}\big)\bigwedge_{1\leq i\leq n}(\sqrt{-1}dz_i\wedge d\bar{z}_i) 
    = \big(\sum_{0\leq i\leq p}\frac{|s|^2}{|z_i|^2} \big)(\sqrt{-1})^{n^2}\Omega_{\mathcal{X}_s}\wedge\overline{\Omega}_{\mathcal{X}_s},
\end{align*}
we have
\begin{align*}
    \tilde{\omega}_s^n = \rho_1\cdot \big(\sum_{0\leq i\leq p}\frac{|s|^2}{|z_i|^2} \big) (\sqrt{-1})^{n^2}\Omega_{\mathcal{X}_s}\wedge\bar{\Omega}_{\mathcal{X}_s},
\end{align*}
where there exists $C_{1,1}>0$ such that $C_{1,1}^{-1}<\rho_1<C_{1,1}$.
Similarly, on ${X}_s\cap \mathcal{U}$, $\mathcal{K}$ has a trivialization 
\begin{align*}
    \Omega_{\mathcal{K}} = (\prod_{0\leq i\leq p} z_i^{a_i})\Omega_{\mathcal{X}_s}.
\end{align*}
And the unnormalized measure associated to the holomorphic $n$-form is
\begin{align*}
    \nu_s = \rho_2\cdot (\prod_{0\leq i\leq p}|z_i|^{2a_i})\cdot (\sqrt{-1})^{n^2} \Omega_{\mathcal{X}_s}\wedge\bar{\Omega}_{\mathcal{X}_s},
\end{align*}
where there exists $C_{1,2}>0$ such that $C_{1,2}^{-1}<\rho_2<C_{1,2}$.

Meanwhile, under our choice of normalization $\min_{i\in I} (a_i)=0$, by \cite[Theorem A, C]{BJ17}, $\int_{X_s}\nu_s \approx \lambda(s)^d$ for $0\leq d\leq n$. Note that $|s|\leq \lambda(s)=|\log|s|| \leq |s|^{-1}$ when $|s|<e^{-1}$. Then by choosing $C_2 = \sum_{0\leq i\leq p}a_i+2$, there exists $C_{1,3}>0$, such that on $X_s\cap \mathcal{U}$,
\begin{align*}
    C_{1,3}^{-1}|s|^{2C_2} \leq \frac{\tilde{\omega}_s^n}{\mu_s} \leq C_{1,3}\cdot |s|^{-2 C_2}.
\end{align*}
Since a neighborhood of $\mathcal{X}_0$ can be covered by finitely many local charts $\mathcal{U}$, the inequality \eqref{volume_ratio_bound_inequality} is proved.
 \end{proof}

\begin{lemma}
    \label{monotone_of_volume_ratio}
    Let $\tilde{\omega}_{s,t} = \lambda(s)\cdot \omega_{s,t}$, $\rho_{s,t} = \frac{\tilde{\omega}_{s,t}^n}{\mu_s}$. Then for any $p> 1$,
    \begin{align*}
        \int_{X_s} \rho_{s,t}^p \;\mu_s
    \end{align*}
    is decreasing along $t\in [0,+\infty)$.
\end{lemma}
\begin{proof}
    The proof is by a direct computation.
    By the parabolic equation \eqref{family_of_parabolic_Monge_Ampere_flow},
    \begin{align*}
        \partial_t\log(\rho_{s,t}) = \Delta_{\omega_{s,t}}\partial_t\varphi_s = \Delta_{\omega_{s,t}}\log(\frac{\tilde{\omega}_{s,t}^n}{\mu_s}) = \Delta_{\omega_{s,t}}\log(\rho_{s,t}).
    \end{align*}
    Then
    \begin{align*}
        \frac{d}{dt}\int_{X_s} \rho_{s,t}^p \; \mu_s &= p\int_{X_{s}}\rho_{s,t}^{p-1}\partial_t\log(\rho_{s,t})\frac{\omega_{s,t}^n}{V_{\omega_s}} \\
        &= -p(p-1)\int_{X_s} \rho_{s,t}^{p-3} |\nabla\rho_{s,t}|^2_{\omega_{s,t}}\frac{\omega_{s,t}^n}{V_{\omega_s}} \leq 0.
    \end{align*}
\end{proof}
Let $q = \frac{p}{p-1}$.
Combining this monotonicity with Lemma \ref{volume_ratio_bound}, we obtain the estimate needed below:
\begin{align}
\label{L^p_bound_of_volume_ratio}
    \Big(\int_{X_s} \rho_{s,t}^p \mu_s \Big)^{\frac{1}{p}} &\leq \Big(\int_{X_s}\rho_{s,0}^p\mu_s\Big)^{\frac{1}{p}} \\ \nonumber
    &\leq \Big(C_1^{p-1} |s|^{-2C_2(p-1)} \int_{X_s} \tilde{\omega}_{s}^n\Big)^{\frac{1}{p}} \\ \nonumber
    &= C_1^{\frac{1}{q}}|s|^{-2C_2 q^{-1}}.
\end{align}

We also need a uniform H\"ormander-Tian type estimate for Fubini-Study metrics. Recall that
we have the embedding $\iota: \mathcal{X}\xrightarrow{|m\mathcal{L}|} \mathbb{P}^N\times \mathbb{D}$, and $m\tilde{\omega}_{s} = \iota^*\omega_{\rm FS}|_{X_s}$.
\begin{lemma}
    \label{Skoda_for_FS}
    There exists a uniform constant $C_3>0$ independent of $s\in\mathbb{D}^*$, such that for any $s\in \mathbb{D}^*$, any $\tilde{u}_s\in \PSH(X_s,\tilde{\omega}_s)$ with $\sup_{X_s}\tilde{u}_s = 0$, we have
    \begin{align}
        \label{Skoda_for_FS_inequality}
        \int_{X_s} e^{\frac{-1}{nm^{n-1}}\tilde{u}_s} \;\tilde{\omega}_s^n \leq C_3 e^{C_3\lambda(s)}.
    \end{align}
\end{lemma}
\begin{proof}
    Since $m\tilde{\omega}_s$ is the pull-back of a Fubini-Study metric, by \cite[Proposition 2.5]{DNGG23} (also see \cite[Section 3]{Zeriahi01}), we have
     \begin{align*}
        \int_{X_s} e^{\frac{-1}{nm^{n-1}}\tilde{u}_s} \;\tilde{\omega}_s^n \leq (4n)^n\cdot \exp\big(-\frac{m}{n}\int_{X_s}\tilde{u}_s \; \tilde{\omega}_s^n\big).        
    \end{align*}
    Since $\lambda(s)^{-1}\tilde{u}_s(\cdot,t)\in \PSH(X_s,\omega_{s})$, by \cite[Proposition 1.1]{Cao26}  and the uniform $L^\infty$-estimate,
    \begin{align*}
        \int_{X_s} e^{-\alpha \lambda(s)^{-1}\tilde{u}_s} \tilde{\omega}_s^n < A_0e^{\alpha C}.
    \end{align*}
    By Jensen's inequality, 
    \begin{align*}
        -\int_{X_s}\tilde{u}_s \; \tilde{\omega}_s^n \leq \alpha^{-1}\log(A_0e^{\alpha C}) \lambda(s).
    \end{align*}
    Thus there exists a uniform constant $C_3>0$ for which inequality \eqref{Skoda_for_FS_inequality} holds.
\end{proof}

\medskip

In the following, we will finish the proof of the inequality \eqref{main_thm_2_uniform_Skoda}. Let $\beta = \frac{1}{n m^{n-1}}$. Let $u_s$ be any function in $\PSH(X_s,\omega_{s,t})$ with $\sup_{X_s}u_s = 0$. Let 
\begin{align*}
    \tilde{v}_s = \lambda(s)\big(u_s+\varphi_s(\cdot,t)-\sup_{X_s}(u_s+\varphi_s(\cdot,t))\big) ,
\end{align*}
which is a function in $\in \PSH(X_s,\tilde{\omega}_s)$ with $\sup_{X_s} \tilde{v}_s = 0$.
By the uniform $L^\infty$-estimate,
$\|\varphi_s\|_{L^\infty}\leq C$. Let $q = \frac{\beta\lambda(s)}{\alpha}$.

When $q>1$, set $p = \frac{q}{q-1}$. Recall that $\rho_{s,t} = \frac{\tilde{\omega}_{s,t}^n}{\mu_s}$. Without loss of generality, we may assume $C_1,C_2,C_3\geq1$. Then
\begin{align*}
    \int_{X_s} e^{-\alpha u_s}\; \frac{\omega_{s,t}^n}{V_{\omega_{\omega_s}}} &= \int_{X_s} e^{-\alpha u_s} \rho_{s,t}\;\mu_s \\
    &\leq \Big(\int_{X_s}e^{-\alpha qu_s}\;\mu_s\Big)^{\frac{1}{q}}\cdot \Big(\int_{X_s}\rho_{s,t}^p\;\mu_s\Big)^{\frac{1}{p}} \\
    &\leq \Big(C_1|s|^{-2C_2}e^{2C\beta \lambda(s)}\int_{X_s} e^{-\beta\tilde{v}_s}\;\tilde{\omega}_s^n\Big)^{\frac{1}{q}}\cdot \Big(\int_{X_s} \rho_{s,t}^p\;\mu_s\Big)^{\frac{1}{p}} \\
    &\leq \Big(C_1C_3 \exp\big((2C_2+2C\beta+C_3)\lambda(s) \big) \Big)^{\frac{1}{q}}\cdot\Big(C_1^{\frac{1}{q}} \exp\big(  2C_2\lambda(s)q^{-1}  \big)\Big)\\
    &\leq C_1^2C_3\exp\big((4C_2+2C\beta+C_3)\frac{\alpha}{\beta}\big) =: C_4,
\end{align*}
where in the second line, we use H\"older's inequality; in the third line, we use the uniform $L^\infty$-estimate and inequality \eqref{volume_ratio_bound_inequality}; in the forth line, we use inequality \eqref{Skoda_for_FS_inequality} and inequality \eqref{L^p_bound_of_volume_ratio}.

When $q\leq 1$, we have $\lambda(s)\leq \frac{\alpha}{\beta}$ and obtain
\begin{align*}
    \int_{X_s} e^{-\alpha u_s}\; \frac{\omega_{s,t}^n}{V_{\omega_{\omega_s}}} &\leq e^{C\alpha}\int_{X_s}e^{-\alpha(u_s+\varphi_s(\cdot,t))}\rho_{s,t}\;\mu_s \\
    &\leq e^{C\alpha}\cdot C_1 e^{2C_2\lambda(s)}\cdot \int_{X_s} e^{-\alpha(u_s+\varphi_s(\cdot,t))}\; \mu_s \\
    &\leq C_1 A_\infty \exp\big(2C\alpha+2C_2\frac{\alpha}{\beta}\big) =: C_5,
\end{align*}
where in the first line, we use the uniform $L^\infty$-estimate; in the second line, we use inequality \eqref{volume_ratio_bound_inequality}; in the third line, we use the uniform $L^\infty$-estimate and \cite[Theorem 1.3]{LiYang24}.

Let $A = \max(C_4, C_5)+1$. Then we have
\begin{align*}
    \int_{X_s} e^{-\alpha u_s}\; \frac{\omega_{s,t}^n}{V_{\omega_{\omega_s}}} <A.
\end{align*}
This concludes the proof of Theorem \ref{main_thm_2}.

\section{Uniform diameter and volume estimates}
\label{Green's_function}

In this section, we prove Theorem \ref{main_thm_3}. 
In \cite{GPSS24,GPSSb24,GPSS23}, estimates for Green function and Sobolev inequality with respect to degenerate metrics were systematically studied and an effective approach was established. Based on \cite{GPSS24}, a uniform version of the estimates were developed in \cite{GV25}.
Following \cite{GPSS24,GV25}, we first establish Green function and local-volume estimates for $\omega_{s,t}$ in the family considered in Theorem \ref{main_thm_2}. We then use these estimates to obtain the uniform diameter bound and the Gromov precompactness result.
\medskip

Let $G_x^{\omega_{s,t}}$ be the Green function with respect to the metric $\omega_{s,t}$ and normalized volume $\frac{\omega_{s,t}^n}{V_{\omega_{s,t}}}$, which satisfies that
\begin{align}
    \begin{cases}
        \frac{1}{V_{s,t}}(\omega_{s,t}+\ddc G_x^{\omega_{s,t}})\wedge \omega_{s,t}^{n-1} = \delta_x , \\
        \int_X G_{x}^{\omega_{s,t}}\frac{\omega_{s,t}^n}{V_{\omega_{s,t}}}=0.
    \end{cases}
\end{align}

\begin{proposition}
    \label{Green_function_estimates}
    Let $\omega_{s,t}$ be the family of K\"ahler metrics considered in Theorem \ref{main_thm_2}, where $s\in\mathbb{D}^*,t\in[0,+\infty]$. Fix $0<r_1<\frac{n}{n-1}$, $0<r_2<\frac{2n}{2n-1}$. Then there exist uniform constants $C_1,C_2,C_3>0$, such that for any $\omega_{s,t}$, we have
    \begin{align}
        \label{Green_function_estimates_inequalities}
        &\sup_{X_s} G_x^{\omega_{s,t}} \leq C_1, \\ \nonumber
        &\int_{X_s}|G_x^{\omega_{s,t}}|^{r_1}\frac{\omega_{s,t}^n}{V_{\omega_{s,t}}} \leq C_2, \\ \nonumber
        &\int_{X_s} |\nabla G_x^{\omega{s,t}}|_{\omega_{s,t}}^{r_2}\frac{\omega_{s,t}^n}{V_{\omega_{s,t}}}\leq C_3.
    \end{align}
\end{proposition}
\begin{proof}
    The proof is essentially the same as the proof of \cite[Lemma 2.1, Proposition 2.2, Theorem 2.3]{GV25}, which is built on \cite[Section 5]{GPSS24}. Therefore, we will only sketch the modifications needed for our family of metrics that satisfy the uniform H\"ormander-Tian inequality.
    \medskip

    {\bf Claim 1:} There exists a uniform constant $C_{1,1}>0$ such that, for every continuous $\omega_{s,t}$-quasi-subharmonic function $u$ satisfying $\Delta_{\omega_{s,t}}u\geq -n$ and $\int_{X_s} u\frac{\omega_{s,t}^n}{V_{\omega_{s,t}}}=0$,  we have
    $\sup_{X_s} u \leq C_{1,1}\cdot (n+\int_{X_s}|u|\frac{\omega_{s,t}^n}{V_{\omega_{s,t}}})$.
    \medskip

    Let $v\in \PSH(X_s,\omega_{s,t})$ be an auxiliary function that satisfies the equation
    \begin{align*}
        \frac{1}{V_{\omega_{s,t}}}(\omega_{s,t}+\ddc v)^n = \frac{1+u^+}{1+M}\frac{\omega_{s,t}^n}{V_{\omega_{s,t}}}
    \end{align*}
    with $\sup_{X_s}v=-1$, where $u^+ = \tilde{\max}(u,0)$ is the convex regularized maximum and $M=\int_{X_s} u^+\frac{\omega_{s,t}^n}{V_{\omega_{s,t}}}$.
    By applying the maximum principle argument to $1+u^+-\epsilon(-v)^{\frac{n}{n+1}}$, where $\frac{\epsilon^{n+1}(\frac{n}{n+1})^n}{(1+\frac{n\epsilon}{n+1})^n} = (1+M)$,
    we have
    $1+u^+\leq \epsilon(-v)^{\frac{n}{n+1}}$.

      Meanwhile, the uniform H\"ormander-Tian inequality \eqref{main_thm_2_uniform_Skoda} implies that, for any fixed $p>1$, there exists a uniform constant $C_{1.2}>0$ such that
    \begin{align*}
        \int_{X_s}(-v)^{p}\frac{\omega_{s,t}^n}{V_{\omega_{s,t}}}\leq C(p)\int_{X_s}e^{-\alpha v}\frac{\omega_{s,t}^n}{V_{\omega_{s,t}}}\leq C_{1,2}.
    \end{align*}
    The right-hand-side of the equation above has a uniform $L^p$-bound. Then by the uniform $L^\infty$-estimate of the Monge-Amp\`ere equation (see for example \cite[Theorem 1.1]{DNGG23} or the proof of Theorem \ref{main_thm_1}), there exists a uniform constant $C_{1,3}>0$ such that $\|v\|_{L^\infty}\leq C_{1,3}$. This proves  {\bf Claim 1}.
    \medskip

    {\bf Claim 2:}
    There exists a uniform constant $C_{1,4}>0$ such that, for any continuous $\omega_{s,t}$-quasi-subharmonic function $u$ satisfying $|\Delta_{\omega{s,t}} u| \leq 1$ and $\int_{X_s} u\frac{\omega_{s,t}^n}{V_{\omega_{s,t}}}=0$, we have $\|u\|_{L^\infty}\leq C_{1,4}$.
    \medskip

    Let $\delta=\frac{1}{4(1+4n^2C_{1,1}^2)^2}$. By {\bf Claim 1}, it suffices to show that if $\int_{X_s}u\frac{\omega_{s,t}^n}{V_{\omega_{s,t}}} = 0$ and $|\Delta_{\omega_{s,t}}u|\leq \delta$, then $M=\int_{X_s}|u|\frac{\omega_{s,t}^n}{V_{s,t}}\leq C_{1,5}$ for some uniform constant $C_{1,5}>0$.

    Let $v\in\PSH(X_s,\omega_{s,t})$ be the solution of 
    \begin{align*}
        \frac{1}{V_{\omega_{s,t}}}(\omega_{s,t}+\ddc v)^n = e^{n\sqrt{\delta'}M^{-1}( v -u)}(1+\frac{1}{n}\Delta_{\omega_{s,t}}u)^n\frac{\omega_{s,t}^n}{V_{\omega_{s,t}}}.
    \end{align*}
    where $\delta' = \|\frac{1}{n}\Delta_{\omega_{s,t}} u\|_{L^\infty(X_s)}$.
    By the comparison principle \cite[Proposition 1.4]{GV25}, $v\leq u$. By the uniform H\"ormander-Tian inequality \eqref{main_thm_2_uniform_Skoda} and the uniform $L^\infty$-estimate for the Monge-Amp\`ere equation, there exists a uniform constant $C_{1,6}>0$ such that
    \begin{align*}
        {\rm Osc}_{X_s}(v) \leq C_{1,6}.
    \end{align*}
    The rest of the proof for {\bf Claim 2} follows exactly the same argument as in the proof of \cite[Proposition 2.2]{GV25}.
    \medskip

    By the same argument as step 1 in the proof of \cite[Theorem 2.3]{GV25}, we have $G_x^{\omega_{s,t}}\leq C_1$ for some uniform constant $C_1>0$.
    \medskip

    We next prove the second inequality of \eqref{Green_function_estimates_inequalities}.
    Let $\mathcal{G}_x^{\omega_{\omega_{s,t}}}=G_x^{\omega_{s,t}}-C_1-1$. Since $\int_{X_s}G_x^{\omega_{s,t}}\frac{\omega_{s,t}^n}{V_{\omega_{s,t}}}=0$ and $\sup_{X_s}G_x^{\omega_{s,t}}\leq C_1$, we have
    \begin{align*}
        \int_{X_s}\mathcal{G}_x^{\omega_{s,t}}\frac{\omega_{s,t}^n}{V_{\omega_{s,t}}}\leq C_{2,1}
    \end{align*}
    for some uniform constant $C_{2,1}>0$.
    Choose $0<\beta<\frac{1}{n}$ such that
    \begin{align*}
        r_1<\sum_{i=0}^{M}\beta^i<\frac{n}{n-1},
    \end{align*}
    where $M>0$ is a finite integer. Let $\beta_i = \sum_{0\leq k\leq i}\beta^k$. Then $\beta_{i+1} = 1+\beta\cdot\beta_i$. Let the $\omega_{s,t}$-quasi-subharmonic function $u$ be the solution of
    \begin{align*}
        \frac{1}{V_{\omega_{s,t}}}(\omega_{s,t}+\ddc u)\wedge\omega_{s,t}^{n-1} = \frac{(-\mathcal{G}_x^{\omega_{s,t}})^{\beta\cdot\beta_i}}{\int_{X_s}(-\mathcal{G}_x^{\omega_{s,t}})^{\beta\cdot\beta_i}\frac{\omega_{s,t}^n}{V_{\omega_{s,t}}}}\frac{\omega_{s,t}^n}{V_{\omega_{s,t}}},
    \end{align*}
    with $\int_{X_s}u\frac{\omega_{s,t}^n}{V_{\omega_{s,t}}}=0$. Assume for the moment that $\|u\|_{L^\infty}\leq C_{2,2}$ for some uniform constant $C_{2,2}>0$ that is also independent of $0\leq i<M$. Then by the equation above and integration by parts, we have
    \begin{align*}
        -C_{2,2}\leq u(x)=\int_{X_s}u\frac{(\omega_{s,t}+\ddc\mathcal{G}_x^{\omega_{s,t}})\wedge\omega_{s,t}^{n-1}}{V_{\omega_{s,t}}}=-\frac{\int_{X_s}(-\mathcal{G}_x^{\omega_{s,t}})^{\beta_{i+1}}\frac{\omega_{s,t}^n}{V_{\omega_{s,t}}}}{\int_{X_s}(-\mathcal{G}_x^{\omega_{s,t}})^{\beta\beta_{i}}\frac{\omega_{s,t}^n}{V_{\omega_{s,t}}}}+C_1+1.
    \end{align*}
    After iterating for $M$ times, we obtain
    \begin{align*}
        \int_{X_s}(-\mathcal{G}_x^{\omega_{s,t}})^{r_1}\frac{\omega_{s,t}^n}{V_{\omega_{s,t}}} \leq \big(\int_{X_s}(-\mathcal{G}_x^{\omega_{s,t}})^{\beta_{M}}\frac{\omega_{s,t}^n}{V_{\omega_{s,t}}}\big)^{\frac{r_1}{\beta_M}} \leq C_2
    \end{align*}
    for some uniform constant $C_2>0$. It remains to prove the uniform $L^\infty$-estimate for $u$.
    Let $v\in\PSH(X_s,\omega_{s,t})$ be the solution of 
    \begin{align*}
        \frac{1}{V_{\omega_{s,t}}}(\omega_{s,t}+\ddc v)^n = \frac{(-\mathcal{G}_x^{\omega_{s,t}})^{n\beta\beta_i}\frac{\omega_{s,t}^n}{V_{\omega_{s,t}}}}{\int_{X_s}(-\mathcal{G}_x^{\omega_{s,t}})^{n\beta\beta_i}\frac{\omega_{s,t}^n}{V_{\omega_{s,t}}}}
    \end{align*}
    with $\sup_{X_s}v = 0$.  Jensen's inequality and the uniform H\"ormander-Tian inequality \eqref{main_thm_2_uniform_Skoda} gives a uniform bound of $\sup_{X_s} v$. Since $\beta<\frac{1}{n}$, the right-hand-side of the equation above has a uniform $L^p$-bound for some $p>1$ with $n\beta p<1$. Then by the uniform H\"ormander-Tian inequality \eqref{main_thm_2_uniform_Skoda} and the uniform $L^\infty$-estimate for the Monge-Amp\`ere equation, there exists a uniform constant $C_{2,3}>0$ such that
    \begin{align*}
        -C_{2,3}\leq v \leq 0. 
    \end{align*}
    Then by the AM-GM inequality and the comparison principle, $u$ is bounded from below by $v$ up to multiplying and adding some uniform constants, which leads to the uniform $L^\infty$-estimate for $u$.
    \medskip

    The proof for the third inequality of \eqref{Green_function_estimates_inequalities} follows Step 4 of the proof of \cite[Theorem 2.3]{GV25}.
\end{proof}

We also have the following Sobolev estimates for the family $\omega_{s,t}$. The proof is exactly the same as that of \cite[Theorem 2.6]{GV25}, which is built on \cite[Theorem 2.1 and Lemma 6.2]{GPSS24}.

\begin{proposition}
    \label{Sobolev_estimates}
    Let $\omega_{s,t}$ be the family of K\"ahler metrics considered in Theorem \ref{main_thm_2}, where $s\in\mathbb{D},t\in[0,+\infty]$. Fix $1<r<\frac{n}{n-1}$. Then there exists a uniform constant $C_1>0$, such that for any $u\in W^{1,2}(X_s)$, we have
    \begin{align}
        \label{Sobolev_estimates_inequalities_1}
        \Big(\int_{X_s}|u-\bar{u}|^{2r}\frac{\omega_{s,t}^n}{V_{\omega_{s,t}}}\Big)^{\frac{1}{r}} \leq C_1\int_{X_s}|\nabla u|_{\omega_{s,t}}^2 \frac{\omega_{s,t}^n}{V_{\omega_{s,t}}},
    \end{align}
    where $\bar{u} = \int_{X_s}u\frac{\omega_{s,t}^n}{V_{\omega_{s,t}}}$.

    In addition, there exists a uniform constant $C_2>0$ such that, for any domain $\Omega\subset X_s$, any $u\in W^{1,2}_c(\Omega)$, we have
    \begin{align}
        \label{Sobolev_estimates_inequalities_2}
        \Big(\int_{\Omega}|u|^{2r}\frac{\omega_{s,t}^n}{V_{\omega_{s,t}}}\Big)^{\frac{1}{r}}\leq C_2 \big(1+\frac{V_{\omega_{s,t}}(\Omega)}{V_{\omega_{s,t}(X_s\setminus \Omega)}}\big)\int_{\Omega}|\nabla u|^2_{\omega_{s,t}}\frac{\omega_{s,t}^n}{V_{\omega_{s,t}}}.
    \end{align}
\end{proposition}

\medskip

\begin{proof}[Proof of Theorem \ref{main_thm_3}]
    Let $d_{\omega_{s,t}}(x,y)$ denote the distance between the points $x,y\in X_s$. Assume that
    \begin{align*}
        {\rm diam}_{\omega_{s,t}}(X_s) = d_{\omega_{s,t}}(x_0,y_0)
    \end{align*}
    for $x_0,y_0\in X_s$. We have
    \begin{align*}
        d_{\omega_{s,t}}(z,y_0) &= \int_{X_s} d_{\omega_{s,t}}(x,y_0)\frac{(\omega_{s,t}+\ddc G^{\omega_{s,t}}_{z})\wedge\omega_{s,t}^{n-1}}{V_{\omega_{s,t}}} \\
        &= \int_{X_s}d_{\omega_{s,t}}(x,y_0)\frac{\omega_{s,t}^n}{V_{\omega_{s,t}}}-\frac{1}{n}\int_{X_s}\langle \nabla d_{\omega_{s,t}},\nabla G^{\omega_{s,t}}_{z}\rangle_{\omega_{s,t}}\frac{\omega_{s,t}^n}{V_{\omega_{s,t}}}
    \end{align*}
    When $z=y_0$, we have 
    \begin{align*}
        \int_{X_s}d_{\omega_{s,t}}(x,y_0)\frac{\omega_{s,t}^n}{V_{\omega_{s,t}}} = \frac{1}{n}\int_{X_s}\langle \nabla d_{\omega_{s,t}},\nabla G^{\omega_{s,t}}_{y_0}\rangle_{\omega_{s,t}}\frac{\omega_{s,t}^n}{V_{\omega_{s,t}}}\leq \frac{1}{n}\int_{X_s} |\nabla G^{\omega_{s,t}}_{y_0}|_{\omega_{s,t}}\frac{\omega_{s,t}^n}{V_{\omega_{s,t}}} \leq C_1
    \end{align*}
    for some uniform constant $C_1>0$. Then
    \begin{align*}
        d_{\omega_{s,t}}(x_0,y_0)\leq C_1+\int_{X_s}|\nabla G^{\omega_{s,t}}_{x_0}|\frac{\omega_{s,t}^n}{V_{\omega_{s,t}}} \leq 2C_1 =: C_d.
    \end{align*}

     We next prove the lower bound of the local volume growth. For any $x\in X_s$, let $\eta\in C^\infty(X_s)$ such that $\eta = 1$ on $\bar{B}_{\omega_{s,t}}(x,\frac{1}{2}r)$, $\eta = 0$ on $X_s\setminus B_{\omega_{s,t}}(x,r)$ and $|\nabla \eta|_{\omega_{s,t}}\leq \frac{C_2}{r}$ for some uniform constant $C_2>0$. First choose $z_1\in\partial B(x,r)$, where $\eta(z_1)=0$. Since $\gamma >2n$, $\frac{\gamma}{\gamma-1}<\frac{2n}{2n-1}$.
    Then
    \begin{align*}
        \int_{X_s}\eta \; d_{\omega_{s,t}}\frac{\omega_{s,t}^n}{V_{\omega_{s,t}}} &= \frac{1}{n}\int_{X_s}\langle \nabla(\eta d_{\omega_{s,t}}), \nabla G_{z_1}^{\omega_{s,t}}\rangle_{\omega_{s,t}}\frac{\omega_{s,t}^n}{V_{\omega_{s,t}}} \\
        &\leq \frac{1}{n}\Big(\int_{X_s}|\nabla G^{\omega_{s,t}}_{z_1}|^{\frac{\gamma}{\gamma-1}}\frac{\omega_{s,t}^n}{V_{\omega_{s,t}}}\Big)^{\frac{\gamma-1}{\gamma}}\cdot \big(\frac{1}{V_{\omega_{s,t}}}{\rm Vol}_{\omega_{s,t}}(B_{\omega_{s,t}}(x,r))\big)^{\frac{1}{\gamma}}\\
        &\leq C_3\cdot \big(\frac{1}{V_{\omega_{s,t}}}{\rm Vol}_{\omega_{s,t}}(B_{\omega_{s,t}}(x,r))\big)^{\frac{1}{\gamma}}.
    \end{align*}
    Next choose $z_2\in\partial B_{\omega_{s,t}}(x,\frac{1}{2}r)$, where $\eta(z_2)=1$. Then
    \begin{align*}
        \frac{1}{2}r &= \int_{X_s}\eta \; d_{\omega_{s,t}}\frac{\omega_{s,t}^n}{V_{\omega_{s,t}}}-\frac{1}{n}\int_{X_s} \langle \nabla(\eta d_{\omega_{s,t}}), \nabla G_{z_2}^{\omega_{s,t}}\rangle_{\omega_{s,t}}\frac{\omega_{s,t}^n}{V_{\omega_{s,t}}} \\
        &\leq C_4 \cdot \big(\frac{1}{V_{\omega_{s,t}}}{\rm Vol}_{\omega_{s,t}}(B_{\omega_{s,t}}(x,r))\big)^{\frac{1}{\gamma}}.
    \end{align*}
    Then inequality \eqref{main_thm_3_volume_bound} is proved by choosing $C_v = (\frac{1}{2C_4})^\gamma$.
    \medskip

    Finally, for any sufficiently small $\epsilon>0$, let $\{x_i\}_{1\leq i\leq M}$ be an $\epsilon$-dense subset, $B_{\omega_{s,t}}(x_i,\frac{1}{2}\epsilon)$ are disjoint geodesic balls. Then 
    \begin{align*}
        M\leq \frac{1}{\min_{1\leq i\leq M}\frac{1}{V_{\omega_{s,t}}}{\rm Vol}_{\omega_{s,t}}(x_i, \frac{1}{2}\epsilon)} \leq \frac{2^\gamma}{C_v \epsilon^\gamma}.
    \end{align*}
    Therefore, the Gromov-Hausdorff convergence statement follows from Gromov's precompactness theorem.
    
\end{proof}


\bibliography{HRF_ref}
	
\end{document}